\documentclass{amsart}[12pt]
\usepackage{amssymb,amsmath}
\usepackage{enumerate}
\newcommand{\Om} {\Omega}
\newcommand{\ep} {\epsilon}
\newcommand{\gm} {\gamma}
\newcommand{\om} {\omega}

\newcommand{\ii }{\infty}

\newcommand{\ol} {\overline}

\newcommand{\dt }{\delta}
\newcommand{\lb }{\lambda}

\newcommand{\al} {\alpha}

\newcommand{\su} {\subset}
\newcommand{\mc}{\mathcal}
\newcommand{\sg}{\sigma}

\newtheorem{teo}{Theorem}[section]
\newtheorem{pro}{Proposition}[section]
\newtheorem{cor}{Corollary}[section]

\newtheorem{lm}{Lemma}[section]

\theoremstyle{definition}
\newtheorem{rem}{Remark}[section]
\newtheorem{df}{Definition}[section]

\title{A Non-commutative Wiener-Wintner Theorem}
\begin{document}
\date{December 3, 2013}
\begin{abstract}
For a von Neumann algebra $\mathcal M$ with a faithful normal tracial state $\tau$ and a positive ergodic homomorpsism $\al : \mc L^1(\mc M,\tau)\to \mc L^1(\mc M,\tau)$ such that $\tau \circ \al = \tau$ and $\al$ does not increase the norm in $\mc M$, we establish a non-commutative counterpart of the classical Wiener-Wintner ergodic theorem.
\end{abstract}

\author{Semyon Litvinov}
\address{Pennsylvania State University\\ Hazleton, PA 18202, USA}
\email{snl2@psu.edu}

\maketitle

\section{Introduction and preliminaries}
The celebrated Wiener-Wintner theorem is by far one of the most deep and fruitful results of the classical ergodic theory. 
It may be stated as follows.
\begin{teo}\label{Theorem 0}
Let $(\Om, \mu)$ be a probability space, and let $T:\Om \to \Om$ be an ergodic measure preserving transformation. Then for any function 
$f\in L^1(\Om,\mu)$ there exists a set $\Om_f$ of full measure such that, given $\om \in \Om_f$, the averages
$$
a_n(f,\lb)(\om)=\frac 1n \sum_{k=0}^{n-1} \lb^k f(T^k\om)
$$
converge for all $\lb \in \Bbb C$ with $| \lb |=1$.
\end{teo}

The aim of this article is to establish a non-commutative extension of Theorem \ref{Theorem 0}. We follow the path of "simple inequality" as it is
outlined in \cite{as}. This means that our argument relies on a non-commutative Van der Corput's inequality. Note that such an inequality was established in \cite{nsz}.

\vskip 5pt
Let $H$ be a Hilbert space, $B(H)$ the algebra of all bounded linear operators in $H$, $\| \cdot \|_{\ii}$ the uniform norm in $B(H)$,
$\Bbb I$ the unit of $B(H)$. Let $\mathcal M\su B(H)$ be a semifinite von Neumann algebra with a faithful normal semifinite trace $\tau$. We denote by $P(\mathcal M)$ the complete lattice of all projections in $\mathcal M$ and set $e^{\perp}=\Bbb I-e$ whenever $e\in P(\mathcal M)$. 

A densely-defined closed operator $x$ in $H$ is said to be {\it affiliated} with the algebra $\mathcal M$ if
$x^{\prime}x\su xx^{\prime}$ for every $x^{\prime}\in B(H)$ such that $x^{\prime}x=xx^{\prime}$ for all $x\in \mathcal M$. 
An operator $x$ affiliated with $\mathcal M$  is called {\it $\tau$-measurable} if for each $\ep >0$ there exists such $e\in P(\mathcal M)$ with $\tau (e^{\perp})\leq \ep$ that the subspace $eH$ belongs to the domain of $x$. (In this case $xe\in \mathcal M$.) Let $\mathcal L=\mathcal L(\mathcal M, \tau)$ be the set of all $\tau$-measurable operators affiliated with the algebra $\mathcal M$. The topology $t_{\tau}$ defined in $\mathcal L$ by the family
$$ 
V(\ep,\dt)=\{ x\in \mathcal L : \|xe\|_{\ii}\leq \dt \text { \ for some } e\in P(\mathcal M) \text { with } \tau (e^{\perp})\leq \ep \}; \ \ep>0, \ \dt>0 
$$
of (closed) neighborhoods of zero is called a {\it measure topology}.

\begin{teo}\label{Theorem 1}
{\rm (\cite{se}; see also \cite{ne})} 
$(\mathcal L, t_{\tau})$ is a complete metrizable topological $*$-algebra.
\end{teo}

\vskip 5pt
For a positive self-adjoint operator $x=\int_{0}^{\infty}\lambda de_{\lambda}$ affiliated with $\mathcal M$ one can define
$$
\tau(x)=\sup_{n}\tau \left (\int_{0}^{n}\lambda de_{\lambda}\right )=\int_{0}^{\infty} \lb d\tau(e_{\lambda}).
$$
If $1\leq p< \infty$, then the non-commutative {\it $L^p$-space} associated with $(\mathcal M, \tau)$ is defined as
$$
\mathcal L^p=\mathcal L^p(\mathcal M,\tau)=\{ x\in \mathcal L: \| x\|_p=(\tau (| x|^p))^{1/p}<\ii \},
$$
where $|x|=(x^*x)^{1/2}$, the absolute value of $x$. Naturally, $\mathcal L^{\ii}=\mathcal M$.

\vskip 5pt
Let $\al : \mathcal L^1\to \mathcal L^1$  be a positive linear map such that $\al (x)\leq \Bbb I$ and $\tau (\al (x))\leq \tau (x)$ for every 
$x\in \mathcal L^1\cap \mathcal M$ with $0\leq x\leq \Bbb I$. If $\al$ is such a map, then, 
as is shown in \cite{y1,y2}, $\| \al (x)\|_p\leq \| x\|_p$ for each $x=x^*\in \mathcal L^1\cap \mathcal M$ 
and all $1\leq p\leq \ii$. Besides, there exist unique continuous extensions $\al :\mathcal L^p\to \mathcal L^p$ 
for every $1\leq p<\ii$ and a unique ultra-weakly continuous extension $\al : \mathcal M\to \mathcal M$. 
This implies that for every $1\leq p\leq \ii$ and $x \in \mathcal L^p$ we have $\| \al(x)\|_p \leq 2\| x\|_p$.

\vskip 5 pt
Let $\Bbb C_1=\{z\in \Bbb C: |z|=1 \}$. If $1\leq p\leq \ii$, $x\in \mathcal L^p$, $\lb \in \Bbb C_1$, we denote
 
\begin{equation}
a_n(x)=\frac 1 n \sum_{k=0}^{n-1}\al^k(x),
\end{equation}
\begin{equation}
a_n(x,\lb)=\frac 1 n \sum_{k=0}^{n-1}\lb^k \al^k(x).
\end{equation}

\vskip 5pt
There are several generally distinct types of "pointwise" (or "individual") convergence in $\mathcal L$ each of which, 
in the commutative case with finite measure, reduces to the almost everywhere convergence. 
We deal with the so-called {\it almost uniform} (a.u.) and {\it bilateral almost uniform} 
(b.a.u.) convergences for which $x_n\to x$ a.u. (b.a.u.) means that for every $\ep>0$ 
there exists such $e\in P(\mathcal M)$ that $\tau(e^{\perp})\leq \ep$ and $\| (x-x_n)e\|_{\ii}\to 0$ ($\| e(x-x_n)e\|_{\ii}\to 0$, respectively). 
Clearly, a.u. convergence implies b.a.u. convergence.

\vskip 5pt
In \cite{y2} the following non-commutative ergodic theorem was established.

\begin{teo}\label{Theorem 2}  
For every $x\in \mc L^1$, the ergodic averages (1)
converge b.a.u. to some $\hat x\in  \mc L^1$.
\end{teo}

B.a.u. convergence of the averages (2) for $x\in \mc L^1$ and a fixed $\lb \in \Bbb C_1$ was proved in \cite{cls}. 

\section{Non-commutative Wiener-Wintner property}

Now we turn
our attention to a study of the  "simultaneous" on $\Bbb C_1$  individual convergrnce of the averages (2). 
We begin with the following definition; see \cite{li}.

\vskip 5pt
\begin{df}
Let $(X, \| \cdot \|)$ be a normed space. A sequence $a_n:X\to \mathcal L$ of additive maps is called {\it bilaterally uniformly equicontinuous
in measure (b.u.e.m.) at $0\in X$} if for every $\ep>0$, $\dt>0$ there exists $\gm>0$ such that for every $x\in X$ with $\| x\|<\gm$ there is 
$e_x\in P(\mathcal M)$ for which
$$
\tau(e_x^{\perp})\leq \ep \ \  \text {and} \ \ \sup_n\| e_xa_n(x)e_x\|_{\ii}\leq \dt.
$$
\end{df}

A proof of the following fact can be found in \cite{li}.

\begin{pro}\label{Proposition 0}
For any $1\leq p<\ii$ the sequence $\{ a_n\}$ given by (1) is b.u.e.m. at $0\in \mathcal L^p$.
\end{pro}

\begin{lm}\label{Proposition 1}
If $1\leq p< \ii$, then, given $\ep>0$, $\dt>0$, there exists $\gm>0$ such that 
for every $x\in \mc L^p$ with $\| x \|_p\leq \gm$ there is $e\in P(\mc M)$ satisfying 
$$
\tau(e^{\perp})\leq \ep \ \ \text{and} \ \ \sup_n\| ea_n(x,\lb)e\|_{\ii}\leq \dt \ \ \text{for all}  \ \ \lb\in \Bbb C_1.
$$
\end{lm}

\begin{proof}
Fix $\ep>0$, $\dt>0$. By Proposition \ref{Proposition 0}, there esists $\gm>0$ such that for each $\| x\|_p<\gm$ it is possible to
find $e\in P(\mc M)$ such that
$$
\tau(e^{\perp})\leq \frac \ep 4 \ \ \text{and} \ \ \sup_n\| ea_n(x)e\|_{\ii}\leq \frac \dt {24}.
$$
Fix $x\in \mc L^p$ with $\| x\|_p<\gm$. We have $x=(x_1-x_2)+i(x_3-x_4)$, where $x_j\in \mc L^p_+$ and 
$\| x_j\|_p\leq \| x\|_p$ for each $j=1,2,3,4$.

\noindent
If $1\leq j\leq 4$, then $\| x_j\|_p<\gm$, so there is $e_j\in P(\mc M)$ safisfying
$$
\tau(e_j^{\perp})\leq \frac \ep 4 \ \ \text{and} \ \ \sup_n\| e_ja_n(x_j)e_j\|_{\ii}\leq \frac \dt {24}.
$$
Let $e=\land_{j=1}^4e_j$. Then we have 
$$
\tau(e^{\perp})\leq \ep \ \ \text{and} \ \ \sup_n\| ea_n(x_j)e\|_{\ii}\leq \frac \dt {24}, \  j=1,2,3,4.
$$

Now, fix $\lb \in \Bbb C_1$. For $1\leq j\leq 4$ denote
$$
a_n^{(R)}(x_j,\lb)=\frac 1 n \sum_{k=0}^{n-1}Re(\lb^k)\al^k(x_j)+a_n(x_j)=\frac 1 n \sum_{k=0}^{n-1}(Re(\lb^k)+1)\al^k(x_j),
$$
$$
a_n^{(I)}(x_j,\lb)=\frac 1 n \sum_{k=0}^{n-1}Im(\lb^k)\al^k(x_j)+a_n(x_j)=\frac 1 n \sum_{k=0}^{n-1}(Im(\lb^k)+1)\al^k(x_j).
$$
Then $0\leq Re(\lb^k)+1\leq 2$ and $\al^k(x_j)\ge 0$ for every $k$ entail
$$
0\leq ea_n^{(R)}(x_j,\lb)e\leq 2ea_n(x_j)e \ \ \text{for all} \ n.
$$
Therefore
$$
\sup_n \left \| ea_n^{(R)}(x_j,\lb)e \right \|_{\ii}\leq \frac \dt {12}
$$
and, similarly,
$$
\sup_n \left \| ea_n^{(I)}(x_j,\lb)e \right \|_{\ii}\leq \frac \dt {12}.
$$
This implies that, given $1\leq j\leq 4$, we have
$$
\sup_n \left \| ea_n(x_j,\lb)e \right \|_{\ii}=
$$
$$
=\sup_n \left \| e\left (a_n^{(R)}(x_j,\lb)+ia_n^{(I)}(x_j,\lb)-a_n(x_j)-ia_n(x_j)\right )e \right \|_{\ii}\leq \frac \dt 4,
$$
and we conclude that
$$
\sup_n \left \| ea_n(x,\lb)e \right \|_{\ii}=
$$
$$
=\sup_n\left \| e\left (a_n(x_1,\lb)-a_n(x_2,\lb)+ia_n(x_3,\lb)-ia_n(x_4,\lb)\right )e\right \|_{\ii}\leq \dt
$$
for every $\lb \in \Bbb C_1$.
\end{proof}

\begin{df}\label{df 2}
Let $1\leq p<\ii$. We say that $x\in \mc L^p$ satisfies {\it Wiener-Wintner (bilaterally Wiener-Wintner) property} 
and we write $x\in WW$ ($x\in bWW$, respectively) if,
given $\ep>0$, there exists a projection $p\in P(\mc M)$ with $\tau(p^{\perp})\leq \ep$ such that the sequence
$$
\{a_n(x,\lb)p\} \  \ ( \{pa_n(x,\lb)p\}, \ \text{respectively})  \ \ \text{converges in} \ \mc M \ \text{for all} \ \lb \in \Bbb C_1.
$$
\end{df}

Note that $WW\su bWW$, while in the commutative case these sets coinside. 

\noindent
Let $(\Om,\mu)$ be a probability space, and let $T:\Om \to \Om$ be a measure preserving transformation. Then $f\in L^1(\Om, \mu)\cap WW$
would imply that for every $m \in \Bbb N$ there exists $\Om_m$ with $\mu (\Om \setminus \Om_m)\leq \frac 1m$ such that the averages
$a_n(f,\lb)(\om)=\frac 1n \sum_{k=0}^{n-1} \lb^k f(T^k\om)$ converge for all $\om \in \Om_m$ and $\lb \in \Bbb C_1$. Then, with
$\Om_f=\cup_{m=1}^\ii \Om_m$, we have $\mu (\Om_f)=1$, while the averages $a_n(f,\lb)(\om)$ converge for all $\om \in \Om_f$ and
$\lb \in \Bbb C_1$.

Therefore Definition \ref{df 2} presents a proper generalization of the classical Wiener-Wintner property; see \cite{as}, p.28. In an attempt to clarify 
what happens in the non-commutative situation without imposing any additional conditions on $\tau$ and $\al$, we suggest the following.
 
\begin{pro}\label{Proposition 7}
Let $1\leq p<\ii$ and $x\in \mc L^p \cap WW$ ($x\in \mc L^p \cap bWW$). Then 

\begin{enumerate}
\item
for every $\lb \in \Bbb C_1$ there is such $x_\lb \in \mc L^p$ that
$$
a_n(x,\lb)\to x_\lb \ \text{a.u.} \ \ (a_n(x,\lb)\to x_\lb \ \text{b.a.u.}, \ \text{respectively}),
$$
\item
if $p\in P(\mc M)$ is such that
$\{ a_n(x,\lb)p\}$ ($\{ pa_n(x,\lb)p\}$) converges in $\mc M$ for all $\lb \in \Bbb C_1$, then, given
$\lb \in \Bbb C_1$ and $\nu>0$, there is a projection $p_\lb\in P(\mc M)$ such that $p_\lb \leq p$, $\tau(p-p_\lb) \leq \nu$, and
$$
\left \|(a_n(x,\lb)-x_\lb)p_\lb \right \|_\ii \to 0 \ \ (\left \|p_\lb(a_n(x,\lb)-x_\lb)p_\lb \right \|_\ii \to 0, \ \text{respectively}).
$$
\end{enumerate}
\end{pro}
\begin{proof} We will provide a proof for the b.a.u. convergence. Same argument is applicable for the a.u. convergence.

\noindent
(1) Let $x\in \mc L^p\cap bWW$ and $\lb \in \Bbb C_1$. Then for every $\ep>0$ there exists $p\in P(\mc M)$ with $\tau(p^\perp)\leq \ep$ for which
$$
\| p(a_m(x,\lb)-a_n(x,\lb))p\|_\ii \to 0 \ \ \text{as} \ m,n\to \ii.
$$
Then, as it is noticed in \cite{cl}, Proposition 1.3, $a_n(x,\lb)\to x_\lb$ b.a.u. for some $x_\lb \in \mc L$,
which clearly implies that $a_n(x,\lb)\to x_\lb$ {\it bilaterally in measure}, meaning that, given $\ep>0$, $\dt>0$, there exists $N\in \Bbb N$
such that for every $n\ge \Bbb N$ there is $e_n\in P(\mc M)$ with $\tau(e_n^\perp)\leq \ep$ satisfying 
$\left \| e_n(a_n(x,\lb)-x_\lb)e_n \right \|_\ii \leq \dt$. Since the measure topology coinsides with the bilateral measure topology on $\mc L$
( \cite{cls}, Theorem 2.2), we have $a_n(x,\lb)\to x_\lb$ in measure. Then, since 
$\| a_n(x,\lb)\|_p\leq 2\| x\|_p$ for every $n$, by Theorem 1.2 in \cite{cls}, $x_\lb \in \mc L^p$.

\vskip 5pt
\noindent
(2) Let $p\in P(\mc M)$ be such that the sequence $\{ pa_n(x,\lb)p\}$ converges in $\mc M$
for all $\lb \in \Bbb C_1$. By (1), given $\lb \in \Bbb C_1$ and $\nu>0$, there is $e_\lb \in P(\mc M)$ with $\tau(e_\lb^\perp)\leq \nu$ such that
$\left \|e_\lb(a_n(x,\lb)-x_\lb)e_\lb \right \|_\ii \to 0$ as $n\to \ii$. Then $p_\lb=p\land e_\lb$ satisfies the required conditions.
\end{proof}

\begin{rem}\label{Remark 1}
It is desirable to have the following: if $x\in WW$ ($x\in bWW$), then, given $\ep>0$, there exists such $p\in P(\mc M)$
with $\tau(p^\perp)\leq \ep$ that $\| (a_n(x,\lb)-x_\lb)p \|_\ii \to 0$ ($\| p(a_n(x,\lb)-x_\lb)p \|_\ii \to 0$, respectively) for all $\lb \in \Bbb C_1$; 
see Remark \ref{Remark 2} below.
\end{rem}

\begin{teo}\label{Theorem 4}
For each $1\leq p<\ii$ the set $X=\mc L^p\cap bWW$ is closed in $\mc L^p$.
\begin{proof}
Take $x$ in the $\| \cdot \|_p-$closure of $X$ and fix $\ep>0$. By Lemma \ref{Proposition 1}, 
one can find sequences $\{ x_m\}\su X$ and $\{ e_m\}\su P(\mc M)$ in such a way that
$$
\tau(e_m^\perp)\leq \frac \ep {3\cdot 2^m}\ \ \ \text{and} \ \ \sup_n\| e_ma_n(x-x_m,\lb)e_m\|_\ii \leq \frac 1 m
$$
for all $m\in \Bbb N$ and $\lb \in \Bbb C_1$. If we let $e=\land_{m=1}^\ii e_m$, then 
$$
\tau(e^\perp)\leq \frac \ep 3\ \ \ \text{and} \ \ \sup_n\| ea_n(x-x_m,\lb)e\|_\ii \leq \frac 1 m,
$$
$m\in \Bbb N$, $\lb \in \Bbb C_1$. Also, since $\{ x_m\} \su bWW$, one can construct $f\in P(\mc M)$ such that
$$
\tau(f^\perp)\leq \frac \ep 3  \ \ \text{and} \ \ \{fa_n(x_m,\lb)f\} \ \ \text{converges in} \ \mc M \ \ \text{for all} \ m\in \Bbb N, \ \lb \in \Bbb C_1.
$$
Next, there exists $g\in P(\mc M)$ with $\tau(g^\perp)\leq \frac \ep 3$ for which $\{\al^k(x)g\}_{k=0}^\ii \su \mc M$ 
so that $\{ga_n(x,\lb)g\}\su \mc M$ for all $\lb \in \Bbb C_1$.
Now, if $p=e\land f\land g$, then we have $\tau(p^\perp)\leq \ep$,
$$
\sup_n\|pa_n(x-x_m, \lb)p\|_\ii \leq \frac 1 m,
$$
$$
\{pa_n(x_m,\lb)p\}  \ \text{converges in} \ \mc M, 
$$
$$
\text{and} \ \{pa_n(x,\lb)p\} \su \mc M 
$$
for all $m\in \Bbb N$ and $\lb \in \Bbb C_1$.

It remains to show that, for a fixed $\lb \in \Bbb C_1$, the sequence $\{pa_n(x,\lb)p\}$ converges in $\mc M$. So, fix $\dt>0$ and pick $m_0$
such that $\frac 1 {m_0}\leq \frac \dt 3$. Since the sequence $\{pa_n(x_{m_0},\lb)p\}$ converges in $\mc M$, there exists $N$ such that 
$$
\left \| p(a_{n_1}(x_{m_0},\lb)-a_{n_2}(x_{m_0},\lb))p\right \|_\ii \leq \frac \dt 3
$$
whenever $n_1,n_2\ge N$.
Therefore, given $n_1,n_2\ge N$, we can write
$$
\left \| p(a_{n_1}(x,\lb)-a_{n_2}(x,\lb))p\right \|_\ii \leq \left \| pa_{n_1}(x-x_{m_0},\lb)p\right \|_\ii +
$$
$$
+\left \| pa_{n_2}(x-x_{m_0},\lb)p\right \|_\ii + \left \| p(a_{n_1}(x_{m_0},\lb)-a_{n_2}(x_{m_0},\lb))p\right \|_\ii \leq \dt.
$$
This implies that the sequence $\{pa_n(x,\lb)p\}$ converges in $\mc M$ for all $\lb \in \Bbb C_1$, hence $x\in X$ and $X$ is closed in $\mc L^p$.
\end{proof}
\end{teo}

Let $\mc K$ be the $\| \cdot \|_2-$closure of the linear span of the set
\begin{equation}
E=\{ x\in \mc L^2: \al(x)=\mu x \ \text{for some} \ \mu \in \Bbb C_1\}.
\end{equation}

\begin{pro}\label{Proposition 3}
$\mc K \su bWW$.
\begin{proof}
By Theorem \ref{Theorem 4}, it is sufficient to show that $\sum_{j=1}^ma_jx_j\in bWW$ whenever $a_j\in \Bbb C$ and $x_j\in E$, $1\leq j\leq m$.
For this, one will verify that $E\su WW$. 

\noindent
If $x\in E$, then $\al(x)=\mu x$, $\mu \in \Bbb C_1$. Fix $\ep>0$ and find $p\in P(\mc M)$ with $\tau(p^\perp)\leq \ep$ such that $xp\in \mc M$. 
Then, given $\lb \in \Bbb C_1$, we have
$$
a_n(x,\lb)=xp\frac 1 n \sum_{k=0}^{n-1}(\lb \mu)^k.
$$
Therefore, since the averages $\frac 1 n \sum_{k=0}^{n-1}(\lb \mu)^k$ converge, 
we conclude that the sequence $\{ a_n(x,\lb)p\}$ converges in $\mc M$, whence $x\in WW$.
\end{proof}
\end{pro}

\section{Spectral characterization of $\mc K^\perp$}\label{3}

The space $\mc L^2$ equipped with the inner product $(x,y)_\tau =\tau(x^*y)$ is a Hilbert space such that $\| x\|_2 =\| x\|_\tau = (x,x)_\tau^{1/2}$, 
$x\in \mc L^2$.

\vskip 5pt
From now on we shall assume that $\tau$ and $\al$ satisfy the following additional conditions: 
$\tau$ is a state, $\al$ is a homomorphism, and $\tau \circ \al =\tau$.
Notice that then $\| \al(x) \|_2=\| x\|_2$ and $| \tau(x) | \leq \| x\|_2$ for every $x\in \mc L^2$.

\begin{pro}\label{Proposition 4} 
 If $x\in \mc L^2$, then the sequence $\{\gm_x(l)\}_{-\ii}^\ii$ given by 
\begin{equation*}
\gm_x(l)=
\begin{cases}
\tau(x^*\al^l(x)), &\text{if \ $l \ge 0$} \\
\ol {\tau(x^*\al^{-l}(x))}, &\text{if \ $ l<0$.}
\end{cases}
\end{equation*}
is positive definite.
\begin{proof}
If $\mu_0, \dots , \mu_m \in \Bbb C$, then, taking into account that positivity of $\al$ implies that $\al(y)^*=\al(y^*)$, $y\in \mc L^2$,  we have
$$
0\leq \left \| \sum_{k=0}^m\mu_k\al^k(x)\right \|_2^2=\left (\sum_{j=0}^m\mu_j\al^j(x), \sum_{i=0}^m\mu_i\al^i(x) \right )_\tau 
=\sum_{i,j=0}^m\mu_i \bar\mu_j\tau(\al^j(x^*)\al^i(x)).
$$
If $i\ge j$, we can write
$$
\tau(\al^j(x^*)\al^i(x))=\tau(\al^j(x^*\al^{i-j}(x)))=\tau(x^*\al^{i-j}(x))=\gm_x(i-j),
$$
and if $i<j$, we have
$$
\tau(\al^j(x^*)\al^i(x))=\ol { \tau(\al^i(x^*)\al^j(x))}=\ol {\tau(x^*\al^{j-i}(x))}=\gm_x(i-j).
$$
Therefore
$$
\sum_{i,j=0}^m\gm_x(i-j)\mu_i\bar \mu_j\ge 0
$$
for any $\mu_0, \dots , \mu_m \in \Bbb C$, hence $\{ \gm_x(l)\}$ is positive definite.
\end{proof}
\end{pro}

Consequently, given $x\in \mc L^2$, by Herglotz-Bochner theorem, 
there exists a positive finite Borel measure $\sg_x$ on $\Bbb C_1$ such that

\begin{equation}
\tau(x^*\al^l(x))=\gm_x(l)=\widehat {\sg_x}(l)=\int_{\Bbb C_1}e^{2\pi ilt}d\sg_x(t), \ \ l=1,2, \dots  \ .
\end{equation}

\begin{lm}\label{Lemma 1}
$\al(\mc K^\perp)\su \mc K^\perp$.

\begin{proof}
Since $\al :\mc L^2\to \mc L^2$ is an isometry, we have $\| \al \|=1$. 
Therefore $\| \al^*\|=1$ as well, so that $\| \al^*(x)\|_2\leq \| x\|_2$, $x\in \mc L^2$.

\noindent
Let $x\in E$, that is, $x\in \mc L^2$ and $\al(x)=\mu x$ for some $\mu \in \Bbb C_1$. Then we have
$$
\| \al^*(x)-\bar \mu x\|_2^2=\| \al^*(x)\|^2_2-\bar \mu (\al^*(x),x)_\tau - \mu (x,\al^*(x))_\tau +\| x \|^2_2\leq 0,
$$
and it follows that $\al^*(x)=\bar \mu x$.

Now, if $y\in \mc K^\perp$, then   $(\al(y),x)_\tau =(y,\al^*(x))_\tau=\bar \mu (y,x)_\tau =0$, which implies that $\al(y) \perp E$, hence $\al(y) \in \mc K^\perp$.
\end{proof}
\end{lm}

\begin{pro}\label{Proposition 5}
If $x\in \mc K^\perp$, then the measure $\sg_x$ is continuous.
\begin{proof}
We need to show that $\sg_x(\{t\})=0$ for every $t\in \Bbb C_1$. It is known (\cite{ka}, p.42) that
$$
\sg_x(\{t\})=\lim_{n\to \ii}\frac 1n \sum_{l=1}^ne^{2\pi ilt} \widehat {\sg_x}(t),
$$
which is equal to
$$
\lim_{n\to \ii} \frac 1n \sum_{l=1}^n e^{2\pi ilt}\tau(x^*\al^l(x))=
\lim_{n\to \ii}\tau \left ( x^* \left ( \frac 1n \sum_{l=1}^n e^{2\pi ilt} \al^l(x) \right ) \right ).
$$
Therefore it is sufficient to verify that

\begin{equation}
\lim_{n\to \ii} \left \|  \frac 1n \sum_{l=1}^n e^{2\pi ilt} \al^l(x) \right \|_2 =0.
\end{equation}
By the Mean Ergodic theorem applied to $\widetilde \al :\mc L^2\to \mc L^2$ given by $\widetilde \al (x)= e^{2\pi it}\al(x)$, 
we conclude that
$$
\frac 1n \sum_{l=1}^n e^{2\pi ilt} \al^l(x) \to \bar x \ \ \text{in} \ \ \mc L^2.
$$
Since $x\in \mc K^\perp$, by Lemma \ref{Lemma 1}, we have $\al^l(x)\in \mc K^\perp$ for each $l$, which implies that $\bar x\in \mc K^\perp$.
Besides,
$$
\al(\bar x)=\| \cdot \|_2-\lim_{n\to \ii} \frac 1n \sum_{l=1}^n e^{2\pi ilt} \al^{l+1}(x)=e^{-2\pi it}\bar x,
$$
so that $\bar x\in \mc K$. Therefore $\bar x=0$, and (5) follows.
\end{proof}
\end{pro}

\section{Non-commutative Van der Corput's inequality}\label{4}
It was shown in \cite{nsz} that the extremely useful Van der Corput's 
"Fundamental Inequality" (see \cite{as}) can be fully extended to any $*-$algebra:

\begin{teo}\label{Theorem 5} \cite{nsz} If $n\ge 1$, $0\leq m\leq n-1$ are integers and $a_0, \dots , a_{n-1}$ are elements of a $*-$algebra, then
$$
\left (\sum_{k=0}^{n-1}a_k^*\right )\left (\sum_{k=0}^{n-1}a_k\right )\leq \frac {n-1+m}{m+1}\sum_{k=0}^{n-1}a_k^*a_k+
\frac {2(n-1+m)}{m+1}\sum_{l=1}^m\frac {m-l+1}{m+1}Re\sum_{k=0}^{n-1}a_k^*a_{k+l}.
$$
\end{teo}

\begin{cor}\label{Corollary 1} If in Theorem \ref{Theorem 5}, $a_0, \dots , a_{n-1}$ are elements of a $C^*-$algebra with the norm $\| \cdot \|$, then
$$
\left \| \sum_{k=0}^{n-1}a_k \right \|^2 \leq \frac {n-1+m}{m+1} \left \| \sum_{k=0}^{n-1}a_k^*a_k \right \|+
\frac {2(n-1+m)}{m+1}\sum_{l=1}^m\frac {m-l+1}{m+1}\left \| \sum_{k=0}^{n-1}a_k^*a_{k+l}\right \|,
$$
which implies that
$$
\left \| \frac 1n \sum_{k=0}^{n-1}a_k \right \|^2 \leq \frac {n-1+m}{(m+1)n} \left \| \frac 1n \sum_{k=0}^{n-1}a_k^*a_k \right \|+
\frac {2(n-1+m)}{(m+1)n}\sum_{l=1}^m\frac {m-l+1}{m+1}\left \| \frac 1n \sum_{k=0}^{n-1}a_k^*a_{k+l}\right \|,
$$
and further
\begin{equation}
\left \| \frac 1n \sum_{k=0}^{n-1}a_k \right \|^2 < \frac 2{m+1} \left \| \frac 1n \sum_{k=0}^{n-1}a_k^*a_k \right \|+
\frac 4{m+1}\sum_{l=1}^m\left \| \frac 1n \sum_{k=0}^{n-1}a_k^*a_{k+l}\right \|.
\end{equation}
\end{cor}

\section{Proof of the main result}\label 5

We will assume now that $\al$ is ergodic on $\mc L^2$, that is, $\al(x)=x$, $x\in \mc L^2$, implies that $x=c\cdot \Bbb I$, $c\in \Bbb C$.

\begin{pro}\label{Proposition 6}
If $x\in \mc L^2$, then $a_n(x)\to \tau(x)\cdot \Bbb I$ a.u.
\end{pro}
\begin{proof}
By the Mean Ergodic theorem, $a_n(x) \to \bar x$ in $\mc L^2$.
Therefore $\al(a_n(x))\to \al (\bar x)$ in $\mc L^2$, so $\al(\bar x)=\bar x$, 
and the ergodicity of $\al$ implies that $\bar x=c(x)\cdot \Bbb I$.
Then, since $\tau $ is also continuous in $\mc L^2$, we have $\tau(a_n(x))\to \tau (\bar x)=c(x)$, 
hence $c(x)=\tau(x)$ because $\tau(a_n(x))=\tau(x)$ for each $n$.
It is known (\cite{jx}, \cite{li}) that $a_n(x)\to \hat x\in \mc L^2$ a.u., which implies that $a_n(x)\to \hat x$ in measure.
Since $\| \cdot \|_2-$convergence entails convergence in measure, 
we conclude that $\hat x =\bar x=\tau(x) \cdot \Bbb I$.
\end{proof}

\begin{lm}\label{Lemma 2}
If $a,b\in \mc L$ and $e\in P(\mc M)$ are such that $ae,be\in \mc M$, then 
$$
(ae)^*be=ea^*be.
$$
\begin{proof}
We have
$$
((ae)^*be)^*=(be)^*ae\su (be)^*(ea^*)^*\su(ea^*be)^*,
$$
which, since  $((ae)^*be)^*\in B(H)$, implies that $((ae)^*be)^*=(ea^*be)^*$, hence the required equality.
\end{proof}
\end{lm}

Now we can prove our main result, a non-commutative Wiener-Wintner theorem.
\begin{teo}\label{Theorem 6}
Let $\mc M$ be a von Neumann algebra, $\tau$ a faithful normal tracial state on $\mc M$. 
Let $\al : \mc L^1\to \mc L^1$ be a positive ergodic homomorphism such that $\tau \circ \al$=$\tau$
and $\| \al(x)\|_\ii \leq \| x\|_\ii$, $x\in \mc M$. Then $\mc L^1 = bWW$, that is,
for every $x\in \mc L^1$ and $\ep>0$ there exists such a projection $p\in P(\mc M)$ that
$$
\tau(p^\perp)\leq \ep \ \ \text{and} \ \ \{ pa_n(x,\lb)p\} \ \text{converges in} \ \mc M  \ \text{for all} \ \lb \in \Bbb C_1.
$$
\begin{proof} Since $\mc L^2$ is dense in $\mc L^1$, $\mc L^2=\mc K \oplus \mc K^\perp$, 
and $\mc K\su bWW$ (Proposition \ref{Proposition 3}), by 
Theorem \ref{Theorem 4}, it remains to show that $\mc K^\perp \su bWW$. (In fact, we will show that  $\mc K^\perp \su WW$.)

So, let $x\in \mc K^\perp$ and fix $\ep>0$. Since $\{ x^*\al^l(x) \}_{l=0}^\ii \su \mc L^2$, due to Proposition \ref{Proposition 6}, 
one can construct a projection $p\in P(\mc M)$ in such a way that
$$
\tau (p^\perp)\leq \ep, \ \ \{ \al^k(x)p\} \su \mc M \ \ \text{for all} \ k,
$$
$$
pa_n(x^*x)p \to \tau(x^*x) p = \| x \|_2 p \ \ \text{in}  \ \mc M, \ \text{and}
$$
$$
pa_n(x^*\al^l(x))p \to \tau(x^*\al^l(x))p = \widehat {\sg_x}(l)p \ \ \text{in} \ \mc M \ \text{for every} \ l.
$$
Now, if $a_k=\lb^k\al^k(x)p$, $k=0,1,2, \dots$, then, employing Lemma \ref{Lemma 2}, we obtain
$$
a_k^*a_{k+l}=\lb^lp\al^k(x^*\al^l(x))p, \ \ k,l=0,1,2, \dots \ .
$$
At this moment we apply inequality (6) to the sequence $\{ a_k\}\su \mc M$ yielding in view of (1) and (2) that
$$
\sup_{\lb \in \Bbb C_1}\left \| a_n(x,\lb)p \right \|_\ii^2\leq \frac 2{m+1}\left \| pa_n(x^*x)p\right \|_\ii+
\frac 4{m+1}\sum_{l=1}^m\left \|pa_n(x^*\al^l(x))p\right \|_\ii.
$$
Therefore, for a fixed $m$, we have
$$
\limsup_n \sup_{\lb \in \Bbb C_1}\left \| a_n(x,\lb)p \right \|_\ii^2
\leq \frac 2{m+1}\| x\|_2^2+\frac 4{m+1}\sum_{l=1}^m\left | \widehat {\sg_x}(l) \right |.
$$
Since the mesure $\sg_x$ is continuous by Proposition \ref{Proposition 5}, 
Wiener's criterion of continuity of positive finite Borel measure (\cite{ka}, p.42) yields
$$
\lim_{m\to \ii} \frac 1{m+1}\sum_{l=1}^m \left | \widehat {\sg_x}(l) \right |^2=0,
$$
which entails
$$
\lim_{m\to \ii} \frac 1{m+1}\sum_{l=1}^m\left | \widehat {\sg_x}(l) \right | =0.
$$
Thus, we conclude that
\begin{equation}
\lim_{n\to \ii} \sup_{\lb \in \Bbb C_1}\left \| a_n(x,\lb)p \right \|_\ii =0,
\end{equation}
whence $x\in WW$.

\end{proof}
\end{teo}

Note that (7) can be referred as non-commutative Bourgain's uniform Wiener-Wintner ergodic theorem.

\begin{rem}\label{Remark 2} As we have noticed (Proposition \ref{Proposition 7}), for a fixed $\lb \in \Bbb C_1$ and every $x\in \mc L^1$, 
the averages $a_n(x,\lb)$ converge b.a.u. to some $x_\lb \in \mc L^1$. It can be verified \cite{ld} that $x_\lb$ is a scalar multiple of $\Bbb I$.
If we assume additionally that $\al$ is {\it weakly mixing} in $\mc L^2$, that is $1$ is its only eigenvalue there, 
then it is easy to see that the b.a.u. limit of $\{ a_n(x, \lb) \}$ with $x\in \mc L^2$ is zero unless $\lb=1$. 
Since $\mc L^2$ is dense in $\mc L^1$, one can employ
an argument similar to that of Theorem \ref{Theorem 4} to show that $a_n(x,\lb) \to 0$ b.a.u. for every $x\in \mc L^1$ if $\lb \ne 1$. Therefore if $\al$
is weakly mixing, we can replace, in Theorem \ref{Theorem 6},
$$
\{ pa_n(x,\lb)p\} \ \text{converges in} \ \mc M  \ \text{for all} \ \lb \in \Bbb C_1.
$$
by
$$
\| pa_n(x,\lb)p \|_\ii \to 0 \ \ \text{if} \ \lb \ne 0\ \ \text{and} \ \ \| p(a_n(x)-x_1)p\|_\ii \to 0 \ \ \text{for some} \ x_1\in \mc L^1;
$$
see Proposition \ref{Proposition 7} and Remark \ref{Remark 1}.
\end{rem}

\end{document}